\documentclass[sigconf, natbib=false]{acmart}

\usepackage[utf8]{inputenc}
\usepackage{enumitem}
\usepackage{array}
\usepackage{multirow}
\usepackage{multicol}
\usepackage{makecell}
\usepackage{mathtools}
\usepackage[labelfont={sc}]{caption}
\usepackage[capitalize]{cleveref}
\usepackage{todonotes}
\usepackage{comment}
\usepackage[algoruled,vlined,english,linesnumbered]{algorithm2e}
\RequirePackage[
  style=acmnumeric, 
  backend=biber,
  natbib=true,
  giveninits=true,
  ]{biblatex}
\renewbibmacro{in:}{}
\AtBeginBibliography{\small}
\AtEveryBibitem{\clearfield{month}} 
\AtEveryBibitem{\clearfield{day}}
\AtEveryBibitem{\clearfield{isbn}}
\AtEveryBibitem{\clearfield{issn}}
\AtEveryBibitem{\clearlist{language}}
\renewbibmacro*{doi+eprint+url}{%
    \printfield{doi}%
    \newunit\newblock%
    \iftoggle{bbx:eprint}{%
        \usebibmacro{eprint}%
    }{}%
    \newunit\newblock%
    \iffieldundef{doi}{%
        \usebibmacro{url+urldate}}%
        {}%
    }
\newcommand{\N}{\mathbb{N}}
\newcommand{\Q}{\mathbb{Q}}

\newcommand{\C}{\mathbb{C}}
\newcommand{\R}{\mathbb{R}}
\newcommand{\Z}{\mathbb{Z}}

\newcommand{\Fp}{\mathbb{F}_p}
\newcommand{\Dp}{\mathfrak{D}_p}
\newcommand{\ps}[1]{[\![ #1 ]\!]}
\newcommand{\ua}{\boldsymbol{\alpha}}
\newcommand{\ub}{\boldsymbol{\beta}}
\newcommand{\um}{\boldsymbol{\mu}}
\newcommand{\ug}{\boldsymbol{\gamma}}

\newcommand{\val}{\text{\rm val}}
\newcommand{\NP}{\text{\rm NP}}
\newcommand{\ev}{\varepsilon}
\newcommand{\nuval}{\val_{p,\nu}}
\newcommand{\zeroval}{\val_{p,0}}
\newcommand{\equivm}{\equiv_{\times}}
\newcommand{\softO}{O^{\sim}}
\newcommand{\denom}{\text{den}}

\definecolor{input}{HTML}{303060}
\definecolor{output}{HTML}{804000}
\definecolor{string}{HTML}{A02020}
\definecolor{ring}{HTML}{A020A0}
\definecolor{function}{HTML}{205080} 
\definecolor{constructor}{HTML}{205080}
\definecolor{method}{HTML}{205080}
\definecolor{keyword}{HTML}{008000}
\definecolor{error}{HTML}{B01010}
\definecolor{comment}{HTML}{606060}

\def\pFq#1#2#3{\mathcal H\left(#1, #2; #3\right)}
\def\pGq#1#2{\mathcal H\left(#1; #2\right)}

\setlist[itemize]{leftmargin=1.5em}

\begin{abstract}
    We discuss algorithms for arithmetic properties of hypergeometric functions.
    Most notably, we are able to
    compute the $p$-adic valuation of a hypergeometric function on any disk
    of radius smaller than the $p$-adic radius of convergence. This we 
    use, building on work of Christol, to determine the set of prime numbers modulo
    which it can be reduced.
    Moreover, we describe an algorithm to find an annihilating polynomial
    of the reduction of a hypergeometric function modulo~$p$.
\end{abstract}

\ccsdesc[500]{Computing methodologies~Symbolic and algebraic manipulation}

\keywords{Algorithms, Hypergeometric Series, Algebraicity, Tropical algebra, Newton polygons}

\acmConference[ISSAC'26]{ACM Conference}{July 2026}{Oldenburg, Germany}

\title[Algorithms for Hypergeometric Functions]{Algorithms for Algebraic and Arithmetic Attributes\\of Hypergeometric Functions}
\author{Xavier Caruso}
\affiliation{
	\institution{CNRS; IMB, Université de Bordeaux}
    \city{Talence}
    \country{France}
	}
\email{xavier@caruso.ovh}
\author{Florian Fürnsinn}
\affiliation{
	\institution{University of Vienna, Faculty of Mathematics}
    \city{Vienna}
    \country{Austria}
	}
\email{florian.fuernsinn@univie.ac.at}

\begin{document}

\SetKwInOut{Input}{input}
\SetKwInOut{Output}{output}
\SetKw{KwOr}{or}
\SetKw{KwAnd}{and}
\SetKw{Continue}{continue}
\SetKw{Pop}{pop}
\SetKw{Append}{append}

\thanks{
    The second named author was funded by a DOC Fellowship (27150) of the
    \href{https://www.oeaw.ac.at/en/}{Austrian Academy of Sciences} at the
    University of Vienna. Further he thanks the French–Austrian project EAGLES
    (ANR-22-CE91-0007 \& FWF grant 
    \href{https://doi.org/10.55776/I6130}{10.55776/I6130}) for financial 
    support.

    The authors thank \href{https://oead.at/en/}{Austria’s Agency for Education
    and Internationalisation (OeAD)} and Campus France for providing funding
    for research stays via WTZ collaboration project/Amadeus project FR02/2024.
}

\maketitle

\section{Introduction}

A \emph{hypergeometric function} with \emph{top parameters}
$\ua \coloneqq (\alpha_1,\ldots, \alpha_n)\in \C^n$ and \emph{bottom parameters} 
$\ub \coloneqq (\beta_1,\ldots, \beta_m) \in \C^m$
is defined as the power series 
\[\pFq{\ua}{\ub}{x} \coloneqq
 \sum_{k=0}^\infty \frac{(\alpha_1)_k\cdots(\alpha_n)_k}
 {(\beta_1)_k\cdots (\beta_m)_k}\cdot x^k\in \C\ps{x},\]
where $(\gamma)_k \coloneqq \gamma (\gamma+1)\cdots (\gamma+k-1)$
denotes the \emph{rising factorial} or \emph{Pochhammer symbol}. 

We restrict ourselves to rational parameters for our investigation; in
particular our hypergeometric functions are elements of $\Q\ps x$.
To simplify the exposition, we will also suppose 
that none of the $\alpha_i$, $\beta_j$ are nonpositive
integers.
We mention nevertheless that all our algorithms extend without
this hypothesis, assuming only that the hypergeometric function
$\pFq{\ua}{\ub}{x}$ is well-defined.

Usually in the literature, hypergeometric functions are normalized differently.
The hypergeometric function ${}_{n}F_m(\ua, \ub; x)$ is defined as
$\pFq{\ua}{\ub'}{x}$, with $\ub'=(\beta_1,\ldots, \beta_m, 1)$, \emph{i.e.}, it
includes an additional parameter $1$ at the bottom. Conversely,
$\pFq{\ua}{\ub}{x} = {}_{n+1}F_m(\ua', \ub; x)$ with 
$\ua' = (\alpha_1,\ldots, \alpha_n, 1$). The classical notation is convenient
to derive the differential equation 
\[\big(x (\vartheta{+}\alpha_1)\cdots (\vartheta{+}\alpha_n) -
 \vartheta (\vartheta{-}\beta_1)\cdots (\vartheta{-}\beta_m)\big)
 \cdot {}_{n}F_m(\ua, \ub; x) = 0,\]
where $\vartheta = x \frac{\mathrm{d}}{\mathrm{d} x}.$
It shows that hypergeometric functions are \emph{D-finite},
\emph{i.e.}, they satisfy a nonzero 
linear differential equation with polynomial
coefficients.
However, for our arguments and algorithms, it proves more 
convenient not to insist on the additional parameter $1$.

Hypergeometric functions play an important role in combinatorics and 
physics, for example the generating functions of many well-known 
sequences, such as the Catalan numbers, are of this shape. At the same 
time, hypergeometric functions serve as test examples for conjectures on 
algebraic and D-finite series. For example, Christol's conjectured
in~\cite{Chr86, Chr90} that \emph{globally bounded} D-finite series
(\emph{i.e.,} D-finite series with positive radius of convergence,
that can be reduced modulo almost all 
primes) can be written as diagonals of multivariate
algebraic power series and then gave the first evidences by studying 
the class of hypergeometric functions. As of 2026, Christol's 
conjecture is still unsolved, even in the hypergeometric case, 
despite recent progress~\cite{BBC+12, BBC+13, BBMW15, AKM20, BY22}.

On a different note, we recall that a power series $f(x)\in K \ps x$ is 
called \emph{algebraic}, if there exists a nonzero polynomial $P[x,y]\in 
K[x, y]$, such that $P(x, f(x))=0$.
Although D-finite power series $f(x)\in \Q\ps x$ are usually not algebraic,
it is frequently the case that their reductions modulo the primes are
(see \cite[Subsection~2.1]{CFV25} for many examples).
For example, if Christol's conjecture on diagonals proves to be true,
Furstenberg's theorem~\cite{Fur67} would imply that the reduction of
any globally bounded D-finite series is algebraic.
In the article~\cite{CFV25}, written jointly with Vargas-Montoya,
we went further in this direction and proposed a conjecture
predicting what
the Galois groups of D-finite functions modulo the primes could be.

Hypergeometric series form a class of examples for which algebraicity
properties are well studied and understood: while there is only
a small set of parameters $(\ua, \ub)$ leading to algebraic series over 
$\Q$~\cite{Sch73, Lan04, Lan11, Err13, Chr86, BH89, Kat90, FY24}, it is
known that reductions modulo $p$, when they exist, are always
algebraic~\cite{Chr86, Chr86a, Var21, Var24} and certain Galois groups
were computed in~\cite[Subsection~3.3]{CFV25}.

\paragraph{Our contributions.}

The aim of the present article is to develop algorithmic tools
for hypergeometric series, focusing particularly on reductions modulo
primes and their algebraicity properties.

To start with, we describe an algorithm to determine
the set of prime numbers, for which a given hypergeometric function
can be reduced (Subsection~\ref{ssec:red}).
Our algorithm heavily relies on routines to compute the $p$-adic
valuation of hypergeometric functions, which themselves find their
roots in Christol's article~\cite{Chr86}, while similar approaches 
were also studied in \cite[Prop.~24]{DRR17}. We develop those routines
in Subsection~\ref{ssec:val} and extend them to the
computation of Newton polygons (Subsection~\ref{ssec:NP}) and 
evaluation of hypergeometric functions at $p$-adic arguments
(Subsection~\ref{ssec:padiceval}).

Then, building on the work of Christol, Vargas-Montoya
and the authors~\cite{Chr86, Chr86a, Var21, Var24, CFV25}, we design an 
algorithm 
to compute an annihilating polynomial of any hypergeometric series modulo 
$p$ (Subsection~\ref{ssec:algebraicity}). We underline that our approach 
gives in addition a new proof of algebraicity without any assumption 
on the parameters $(\ua, \ub)$, nor on the prime~$p$ (aforementioned
references often exclude certain cases for simplicity).
In the same spirit, we mention that, while most of the questions considered 
in this paper simplify when the prime number $p$ is chosen large enough 
with respect to the absolute value
of the parameters and their common denominator, we especially take care
to treat small primes as well. We believe that it is important for 
applications, given that
large primes usually lead to huge outputs, \emph{e.g.} huge annihilating
polynomials, from which it looks more difficult to extract relevant information.

All the algorithms discussed in this article have been implemented in
SageMath~\cite{github,CF26-sage} and can be tested online at the URL:

\noindent
{\color{magenta}%
\url{https://xavier.caruso.ovh/notebook/hypergeometric-functions/}}

\section{Valuations (the $p$-adic picture)} \label{sec:val}

In this section, we fix a prime number $p$ together with
two tuples of parameters
$\ua \coloneqq (\alpha_1,\ldots, \alpha_n)\in \Q^n$ and
$\ub \coloneqq (\beta_1,\ldots, \beta_m) \in \Q^m$. 
The aim of this section is to study the behavior of the $p$-adic
valuation, denoted $\val_p$, of the coefficients of associated hypergeometric series 
$\pFq{\ua}{\ub}{x}$.

To start with, we observe that, if $\gamma$ is a rational number
with $\val_p(\gamma) < 0$, one has $\val_p((\gamma)_k) = k \cdot \val_p(\gamma)$.
Hence the valuations of the coefficients of $\pFq{\ua}{\ub}{x}$
are directly related to the valuations of the coefficients of 
$\pFq{\ua'}{\ub'}{x}$ where the new parameters $\ua'$ and
$\ub'$ are obtained from $\ua$ and $\ub$ by removing the $\alpha_i$
and $\beta_j$ with negative $p$-adic valuation.
For this reason, we will always assume in what follows that
$\val_p(\alpha_i)$ and $\val_p(\beta_j)$ are all nonnegative.
We set $h(x) \coloneq \pFq{\ua}{\ub}{x}$ and we write
$h_k$ for the coefficient of $h(x)$ in $x^k$.

\subsection{Zigzag functions}
\label{ssec:zigzag}

We fix a positive integer $r > 0$.
Let $w_r : \N \to \N$ be the sequence defined by
$w_r(0) = 0$ and the recurrence relation
\begin{align*}
w_r(k+1) = w_r(k) 
& + \big|\{ 1 \leq i \leq n : k + \alpha_i \equiv 0 \bmod{p^r} \}\big| \\
& - \big|\{ 1 \leq j \leq m : k + \beta_j \equiv 0 \bmod{p^r} \}\big|.
\end{align*}

We follow Kedlaya's terminology, see for example~\cite{Ked22}, and call 
$w_r$ a \emph{zigzag function}.
We start by noticing that
\begin{equation}
\label{eq:wrperiodic}
\forall k \geq 0, \quad w_r(k + p^r) = w_r(k) + (n - m).
\end{equation}
Besides, the function $w_r$ does not vary often. More precisely, we
define the set
$\Gamma \coloneq \{1, \alpha_1, \ldots, \alpha_n, \beta_1, \ldots, \beta_m\}$
and denote by $s$ its cardinality. Let also
$$0 = \xi_{r,0} \leq \xi_{r,1} \leq \cdots
    \leq \xi_{r,s-1} < p^r$$
be the reductions modulo $p^r$ of $1 - \gamma$, for $\gamma \in \Gamma$,
sorted in ascending order.
We prolong the sequence $(\xi_{r,i})_i$ by repeating the same
values translated by $p^r$, $2p^r$, \emph{etc.} Formally, we
set $\xi_{r,i+s} = \xi_{r,i} + p^r$ for $i \geq 0$.
Then, on each interval $I_{r,i} := [\xi_{r,i}, \xi_{r,i+1})$,
the sequence $w_r$ is constant.

We note that, the reductions modulo $p^r$ of a rational number $x$
with denominator coprime with $p$, can be read off on its $p$-adic
expansion. More precisely, if
$$x = x_0 + x_1 p + x_2 p^2 + \cdots \in \mathbb Z_p$$
we have $x \bmod p^r = x_0 + x_1 p + \cdots + x_{r-1} p^{r-1}$.
If we assume moreover that $x$ is not a nonnegative integer, its
$p$-adic expansion is not finite, which implies
that $x \bmod p^r$ goes to infinity when $r$ grows.
More precisely, using that the sequence $(x_i)$ is ultimately
periodic (since $x$ is a rational number), we find that 
$x \bmod p^r$ grows at least linearly in $p^r$.
As a consequence, we derive that $\xi_{r,1} \geq c p^r$ for some
constant $c$ depending only on $\ua$ and $\ub$.

\begin{lemma}
\label{lem:vpak}
For $k \geq 0$, we have
$\val_p(h_k) = \sum_{r=1}^\infty w_r(k)$.
\end{lemma}

The proof of the lemma follows the standard argument used to prove
Legendre's formula giving the $p$-adic valuation of a factorial (see
also \cite[Equation~(6)]{Chr86}).

Noticing that $w_r$ vanishes on the interval $[0, \xi_{r,1})$,
we derive from the estimation $\xi_{r,1} \geq c p^r$ that the
the sum of \cref{lem:vpak} is finite for any given $k$; more
precisely it contains at most $\log_p(k) + O(1)$ terms.

\begin{proposition}
\label{prop:padicradius}
We have $\lim_{k \to \infty} \frac{\val_p(h_k)} k = \frac{n - m}{p-1}$.
\end{proposition}

\begin{proof}
For $k < p^r$, we have $-m \leq w_r(k) \leq n$.
Using Equation~\eqref{eq:wrperiodic}, this implies that there exists
some constant~$M$, such that
$$\textstyle \left|w_r(k) - k {\cdot} \frac{n - m}{p^r}\right| \leq M$$
for arbitrary $k$. Summing over $r$ and using a geometric sum and the
triangle inequality, we then find
$$\textstyle \left|\val_p(h_k) - k {\cdot} \frac{n - m}{p-1} + 
k{\cdot}\frac{n - m}{p^s(p-1)} \right| \leq M s,$$
where $s = \log_p(k) + O(1)$ is the number of summands involved in
the sum of \cref{lem:vpak}.
The term $k{\cdot} \frac{n - m}{p^s(p-1)}$ then remains bounded
when $k$ grows, proving that 
$$\textstyle \left|\val_p(h_k) - k{\cdot}\frac{n - m}{p-1} \right| 
\leq M \log_p(k) + O(1).$$
The proposition follows.
\end{proof}

\subsection{Drifted valuations}
\label{ssec:val}

The valuation of the hypergeometric series $h(x)$
is defined as the minimum of $\val_p(h_k)$ when $k$ varies.
More generally, we define its \emph{$\nu$-drifted valuation} by
$$\nuval (h(x))
  \coloneq \min_{k\geq 0} \, \val_p(h_k) + \nu k.$$
Working with drifted valuations 
is important to handle smoothly the reduction to
parameters whose denominators are not divisible by $p$,
and will be also crucial when
we will study Newton polygons in Subsection~\ref{ssec:NP}.

\subsubsection{Recurrence over the tropical semiring}
\label{sssec:recurrence}

It follows from \cref{prop:padicradius} that
$\nuval(h(x)) = -\infty$
when $\nu < \nu_0 \coloneq \frac{m - n}{p-1}$. From now on, we will
then assume that $\nu \geq \nu_0$.
For $r\in \N$ we introduce the drifted partial sums
$$\textstyle
\sigma_r : \N \to \Q, \quad k \mapsto \nu k + \sum_{s=1}^r w_s(k).$$
It directly follows from \cref{lem:vpak} that
$$\forall k \geq 0, \quad 
  \val_p(h_k) + \nu k = \lim_{r \to \infty} \sigma_r(k).$$
Besides, the sequence $\sigma_r$ satisfies the periodicity condition
\begin{equation}
\label{eq:sigmarperiodic}
\forall k \geq 0, \quad
  \sigma_r(k + p^r) = \sigma_r(k) + (\nu - \nu_0) p^r + \nu_0.
\end{equation}

We recall that we have defined in Subsection~\ref{ssec:zigzag} the
numbers $\xi_{r,i}$ and the intervals $I_{r,i} = [\xi_{r,i},
\xi_{r,i+1})$, in such a way that the function $w_r$ is constant
on each $I_{r,i}$.
We let $\mu_{r,i}$ denote the minimum of $\sigma_r$ on this interval
(with the convention that $\mu_{r,i} = +\infty$ if $I_{r,i}$ is empty).
We are going the prove that the $\mu_{r,i}$ are subject
to recurrence relations which allows to compute them recursively.
First of all, with respect to the variable $i$, we have the
relation
\begin{equation}
\label{eq:muperiodicity}
\mu_{r, i+s} = \mu_{r,i} + (\nu - \nu_0) p^r + \nu_0
\end{equation}
which follows directly from Equation~\eqref{eq:sigmarperiodic}.
Hence the knowledge of the first $s$ terms of the sequence 
$(\mu_{r,i})_{i \geq 0}$ is enough to
reconstruct the whole sequence.

We now move to the variable $r$. We first note that
the $\mu_{1,i}$ are easily computed since $\sigma_1$ is affine
on all intervals $I_{1,i}$.
The key observation to go from $r{-}1$ to $r$ is that each
$\xi_{r,i}$, being the reduction modulo $p^r$ of one of the
$1{-}\gamma$ for $\gamma \in \Gamma$, also appears in the sequence 
$(\xi_{r-1, j})_j$. In other words, there exists an index
$j_{r,i}$ such that $\xi_{r,i} = \xi_{r-1, j_{r,i}}$.
Setting $J_{r,i} = \Z \cap [j_{r,i}, j_{r,i+1})$, we deduce that
$I_{r,i}$ is the disjoint union of the $I_{r-1,j}$ for $j$
varying in $J_{r,i}$ and then
\begin{equation}
\label{eq:recmu}
\mu_{r,i} = w_r(\xi_{r,i}) + \min_{j \in J_{r,i}} \mu_{r-1,j}.
\end{equation}

It is convenient to reformulate what precedes in the language of
tropical algebra. We let $\mathcal T$ denote the tropical semiring
$\Q \sqcup \{+\infty\}$ equipped with the operations $\oplus = \min$
and $\odot = +$.
We set $\um_r = (\mu_{r,0}, \ldots, \mu_{r,s-1})$ and view it as a
row vector over $\mathcal T$.
Then Equation~\eqref{eq:recmu} translates to a relation of the form
\begin{equation}
\label{eq:recmutropical}
\um_r = \um_{r-1} \odot T_r
\end{equation}
where $T_r$ is a square matrix over $\mathcal T$ of size $s$,
which is explicit in terms of the functions $w_r$.

\subsubsection{Halting criterion}
\label{sssec:halting}

Equation~\eqref{eq:recmutropical} gives an efficient recursive method
to compute the $\mu_{r,i}$.
Besides, remembering that $\xi_{r,1}$ goes to infinity when $r$
grows, we find that 
$\nuval(h(x)) = \lim_{r \to \infty} \mu_{r,0}$.

It then only remains to find a criterion to detect when the limit
is attained.

Let $d$ denote a common denominator of the elements of $\Gamma$,
and let $e$ be the multiplicative order of $p$ modulo $d$.

\begin{lemma}
\label{lem:halt}
If $r$ fulfills the three following requirements
\begin{itemize}
\item 
  $r > e + \log_p \max \big\{1, 
  |1{-}\gamma_1|, \ldots, |1{-}\gamma_{s-1}|\big\}$,
\item $p^r (\nu - \nu_0) + \nu_0 \geq me$
\item $\mu_{r,i} \geq \mu_{r,0} + me$ 
  for all $i \in \{1, \ldots, s{-}1\}$,
\end{itemize}
then $\mu_{r',0} = \mu_{r,0}$ for all $r' \geq r$.
\end{lemma}

The proof of \cref{lem:halt} uses the following result.

\begin{lemma}
\label{lem:periodicexpansion}
Let $x \in \Z_{(p)}$ and let $x = \sum_{r \geq 0} x_r p^r$ be its
$p$-adic expansion, with $x_r \in \{0, 1, \ldots, p{-}1\}$ for all
$r$.
Let $e$ be the multiplicative order of $p$ modulo the denominator
of $x$. Then $x_{r+e} = x_r$ for all $r > e + \max(0, \log_p |x|)$.
\end{lemma}

\begin{proof}
Set $d \coloneq p^e - 1$.
It follows from the definition of $e$ that $dx$ is an integer.
Therefore we can write $x = a + \frac b d$ with $a, b \in \Z$
and $0 \leq b < d$.
If $b_0, \ldots, b_{e-1}$ are the digits in base $p$ of $b$, the
$p$-adic expansion of $\frac b d$ is
$\sum_{i=0}^\infty b_{i \bmod e} \cdot p^i$.
Hence, it is periodic (from the start) of period $e$.

We first assume that $a \geq 0$. Then its $p$-adic expansion is
finite and has $1 + \lfloor \log_p a \rfloor$ digits. Moreover,
while performing the addition $a + \frac b d$, the last carry
can move at most by $e$ digits given that $b_0, \ldots, b_{e-1}$
cannot be all equal to $p{-}1$. The lemma follows in this case.

The case $a = -1$ is similar by writing down the subtraction.

Finally, if $a < -1$, we replace $x$ by $-1{-}x$. This has the 
effect of replacing every digit $x_r$ of $x$ by $p{-}1 - x_r$,
which does not change the periodicity properties. So, we are
back to the case $a \geq 0$.
\end{proof}

\begin{proof}[Proof of Lemma~\ref{lem:halt}]
Given Equation~\eqref{eq:muperiodicity}, 
the second and third requirements together imply that
$\mu_{r,i} \geq \mu_{r,0} + me$ for all $i \geq 1$.
Therefore, using Equation~\eqref{eq:recmu}, we find
$\mu_{r+1,0} = \mu_{r,0}$ (since $0 \in J_{r+1,0}$) 
and $\mu_{r+1,i} \geq \mu_{r,0} + m(e{-}1)$
for all $i \geq 1$ (since $w_{r+1}(\xi_{r+1}, i) 
\geq -m$ for all $i$).
Repeating the same argument, we obtain by induction on~$r'$ that,
$\mu_{r',0} = \mu_{r,0}$ and
$$\forall i \geq 1, \quad
\mu_{r',i} \geq \mu_{r,i} + m(e+r-r').$$
for $r' \in \{r{+}1, \ldots, r{+}e\}$.
In order to continue the induction beyond $r{+}e$, we observe
that the first assumption and \cref{lem:periodicexpansion} ensure 
that the values $1{-}\gamma \bmod p^r$, for $\gamma \in \Gamma$, are sorted 
in the same way as $1{-}\gamma \bmod p^{r+e}$.
Thus, the values $\xi_{i,r}$ and $\xi_{i,r+e}$ correspond
to the same parameter for all $i$. 
We deduce that $\xi_{r+e,i} \geq \xi_{r,i} + p^r$ for 
$i < s$.
Applying Equation~\eqref{eq:recmu} with $r{+}1, \ldots, r{+}e$
as before, we find
\begin{align*}
\mu_{r+e,i} 
 & \geq \mu_{r,i} + (\nu - \nu_0) p^r + \nu_0 - me \\
 & \geq \mu_{r,0} + me = \mu_{r+e, 0} + me.
\end{align*}
Therefore, the requirements of the lemma are also fulfilled
with $r$ replaced by $r{+}e$, and the induction can go on.
\end{proof}

When $\nu > \nu_0$, it follows from \cref{prop:padicradius}
that the $\mu_{r,i}$ with $i > 0$ grow at least linearly in 
$p^r$. The condition of \cref{lem:halt} will then be rapidly
fulfilled.

On the contrary, when $\nu = \nu_0$, we rely on the following
lemma.

\begin{lemma}
\label{lem:Trperiodic}
We assume $\nu = \nu_0$ and let $r_0$ be the smallest integer
greater than
$e + \log_p \max \big\{1, 
  |1{-}\gamma_1|, \ldots, |1{-}\gamma_{s-1}|\big\}.$
Then the sequence $(T_r)_{r \geq r_0}$ is periodic of period $e$.
\end{lemma}

\begin{proof}
Let $r \geq r_0$. We have seen in the proof of
\cref{lem:halt} that $\xi_{r,i}$ and 
$\xi_{r+e,i}$ correspond to the same (top or bottom) 
parameter for all $i$.
It follows that 
$w_{r+e}(\xi_{r+e,i}) = w_r(\xi_{r,i})$ and
$J_{r+e, i} = J_{r,i}$ for all $i$. Therefore $T_{r+e} = T_r$
by Equation~\eqref{eq:recmu}.
\end{proof}

We set
$T \coloneq T_{r_0+1} \odot T_{r_0+2} \odot \cdots \odot T_{r_0+e}
  \in M_s(\mathcal T)$.
Then $\um_{r_0 + \ell e} = \um_{r_0} \odot T^{\odot \ell}$ for all $\ell \geq 0$.
Noticing in addition that the sequence $(\mu_{r,0})_{r \geq 0}$ is
nonincreasing, we find that $\nuval(h(x))$
is the first coordinate of $\um_{r_0} T^+$ where
\begin{equation}
\label{eq:weakclosure}
T^+ \coloneq T + T^{\odot 2} + T^{\odot 3} + \cdots
\end{equation}
is the so-called \emph{weak transitive closure} of $T$
(see \cite[\S 1.6.2]{But10}).
The latter can be efficiently computed using the Floyd--Warshall algorithm
(see \cite[Algorithm~1.6.21]{But10}), which then completes our algorithm.

We underline nonetheless that the sum of Equation~\eqref{eq:weakclosure} 
may diverge. This is however not an issue; indeed, this case is detected 
by the Floyd--Warshall algorithm and it corresponds to the situation
where $\nuval(h(x)) = -\infty$.
Hence, we can conclude in all cases.

\theoremstyle{acmdefinition}
\newtheorem{remark}[theorem]{\sc Remark}
\begin{remark}
\label{rem:smallest}
We can adapt the previous algorithm so that it computes in addition
the smallest integer $k$ such that $\val_p(h_k) + \nu k = \nuval(h(x))$.
For doing this, we compute the integers
$$k_{r,i} \coloneq
  \min \big\{ k \in I_{r,i} : \sigma_r(h_k) = \mu_{r,i} \big\}$$
at the same time as the $\mu_{r,i}$
while running the algorithm. We recover $k$ at the end of the computation
using $k = \lim_{r \to \infty} k_{r,0}$.
\end{remark}

\subsubsection{An example}
\label{sssec:example}

We consider the parameters $\ua = \big(\frac 1 3, \frac 4 3\big)$ and
$\ub = \big(\frac 2 3, 1\big)$, set $h(x)=\mathcal H (\ua, \ub; x)$
and want to compute $v_{p,0}(h(x))$.

We start with $p = 7$. We then have $\nu_0 = 0$ and we take $\nu = 0$
as well.
Given that $p \equiv 1 \pmod 3$, we find that the
reductions of the $1{-}\alpha_i$ and $1{-}\beta_j$ modulo $p^r$ are
$$\textstyle
\xi_{r,0} = 0, \quad
\xi_{r,1} = \frac{p^r - 1} 3, \quad  
\xi_{r,2} = \frac{p^r + 2} 3, \quad  
\xi_{r,3} = \frac{2p^r + 1} 3$$      
corresponding to the parameters $1$, $\frac 4 3$, $\frac 1 3$ and $\frac 2 3$
respectively. The zigzag function $w_r$ is $p^r$-periodic and 
\begin{itemize}
\item it takes the value $0$ on the interval $[0, \xi_{r,1})$, 
\item it takes the value $1$ on the interval $[\xi_{r,1}, \xi_{r,2})$,
\item it takes the value $2$ on the interval $[\xi_{r,2}, \xi_{r,3})$,
\item it takes the value $1$ on the interval $[\xi_{r,3}, p^r)$.
\end{itemize}
\begin{figure*}
\begin{tikzpicture}[scale=0.3]
\draw[thin, black!50] (0,0)--(49.5,0);
\draw[dotted, black!50] (49.5,0)--(50.5,0);
\node[left, scale=0.9] at (-1,0) { $\xi_{r, i}$: };
\fill (0,0) circle (1.5mm);
\node[scale=0.7] at (0,-1) { $0$ };
\fill (2,0) circle (1.5mm);
\fill (3,0) circle (1.5mm);
\fill (5,0) circle (1.5mm);
\begin{scope}[xshift=7cm]
\fill (0,0) circle (1.5mm);
\node[scale=0.7] at (0,-1) { $p^{r}$ };
\fill (2,0) circle (1.5mm);
\fill (3,0) circle (1.5mm);
\fill (5,0) circle (1.5mm);
\end{scope}
\begin{scope}[xshift=14cm]
\fill (0,0) circle (1.5mm);
\node[scale=0.7] at (0,-1) { $2p^{r}$ };
\fill (2,0) circle (1.5mm);
\fill (3,0) circle (1.5mm);
\fill (5,0) circle (1.5mm);
\end{scope}
\begin{scope}[xshift=21cm]
\fill (0,0) circle (1.5mm);
\node[scale=0.7] at (0,-1) { $3p^{r}$ };
\fill (2,0) circle (1.5mm);
\fill (3,0) circle (1.5mm);
\fill (5,0) circle (1.5mm);
\end{scope}
\begin{scope}[xshift=28cm]
\fill (0,0) circle (1.5mm);
\node[scale=0.7] at (0,-1) { $4p^{r}$ };
\fill (2,0) circle (1.5mm);
\fill (3,0) circle (1.5mm);
\fill (5,0) circle (1.5mm);
\end{scope}
\begin{scope}[xshift=35cm]
\fill (0,0) circle (1.5mm);
\node[scale=0.7] at (0,-1) { $5p^{r}$ };
\fill (2,0) circle (1.5mm);
\fill (3,0) circle (1.5mm);
\fill (5,0) circle (1.5mm);
\end{scope}
\begin{scope}[xshift=42cm]
\fill (0,0) circle (1.5mm);
\node[scale=0.7] at (0,-1) { $6p^{r}$ };
\fill (2,0) circle (1.5mm);
\fill (3,0) circle (1.5mm);
\fill (5,0) circle (1.5mm);
\end{scope}
\fill (49,0) circle (1.5mm);
\node[scale=0.7] at (49,-1) { $p^{r+1}$ };
\draw[thin, black!50] (0,-2)--(49.5,-2);
\draw[dotted, black!50] (49.5,-2)--(50.5,-2);
\node[left, scale=0.9] at (-1,-2) { $\xi_{r+1, i}$: };
\fill (0,-2) circle (1.5mm);
\fill (16,-2) circle (1.5mm);
\fill (17,-2) circle (1.5mm);
\fill (33,-2) circle (1.5mm);
\fill (49,-2) circle (1.5mm);
\draw (-0.2,0.5)--(0,1)--(15.3,1)--(15.5,0.5);
\node[scale=0.8] at (8,2) { $J_{r,0} = \{0, 1, \ldots, 8\}$ };
\draw (15.8,0.5)--(16,1)--(16.3,1)--(16.5,0.5);
\node[scale=0.8] at (16.5,2) { $J_{r,1} = \{9\}$ };
\draw (16.8,0.5)--(17,1)--(32.3,1)--(32.5,0.5);
\node[scale=0.8] at (25,2) { $J_{r,2} = \{10, 11, \ldots, 18\}$ };
\draw (32.8,0.5)--(33,1)--(48.3,1)--(48.5,0.5);
\node[scale=0.8] at (41,2) { $J_{r,3} = \{19, 20, \ldots, 27\}$ };
\end{tikzpicture}
\caption{\rm 
The sets $J_{r,i}$ for $\ua = \big(\frac 1 3, \frac 4 3\big)$,
$\ub = \big(\frac 2 3, 1\big)$ and $p = 7$}
\label{fig:Jri}
\end{figure*}
From this, we infer the $J_{r,i}$ (see Figure~\ref{fig:Jri}) and 
the matrices $T_r \in M_4(\mathcal T)$; they are independent of $r$
and given by:
$$T_r = \left(\begin{matrix}
0 & +\infty & 2 & 1 \\
0 & 1 & 2 & 1 \\
0 & +\infty & 2 & 1 \\
0 & +\infty & 2 & 1 
\end{matrix}\right).$$
The Floyd--Warshall algorithm gives $T^+ = T_r$.
We finally compute $\um_1 = (0, 1, 2, 1)$ and conclude that $\zeroval(h(x))$
is the first coordinate of the vector $\um_1 \odot T^+ = (0, 2, 2, 1)$,
\emph{i.e.} $\zeroval(h(x)) = 0$.

We now take $p = 11$ and continue with $\nu = 0$.
In this case, the ``$\xi$-ordering'' depends on the parity of $r$.
Precisely, if $r$ is even,
$$\textstyle
\xi_{r,0} = 0, \quad
\xi_{r,1} = \frac{p^r - 1} 3, \quad  
\xi_{r,2} = \frac{p^r + 2} 3, \quad  
\xi_{r,3} = \frac{2p^r + 1} 3$$      
corresponding to the parameters $1$, $\frac 4 3$, $\frac 1 3$ and $\frac 2 3$
respectively as before. However, if $r$ is odd, we have
$$\textstyle
\xi_{r,0} = 0, \quad
\xi_{r,1} = \frac{p^r + 1} 3, \quad   
\xi_{r,2} = \frac{2p^r - 1} 3, \quad  
\xi_{r,3} = \frac{2p^r + 2} 3$$       
corresponding to the parameters $1$, $\frac 2 3$, $\frac 4 3$, $\frac 1 3$
in this order.
Therefore, the matrix $T_r$ also depends on the parity of $r$; a calculation
gives
$$T_{2r'} = \left(\begin{matrix}
0 & +\infty & 2 & 1 \\
0 & +\infty & 2 & 1 \\
0 & 1 & 2 & 1 \\
0 & +\infty & 2 & 1 
\end{matrix}\right)
\quad ; \quad
T_{2r'+1} = \left(\begin{matrix}
0 & -1 & +\infty & 1 \\
0 & -1 & 0 & 1 \\
0 & -1 & +\infty & 1 \\
0 & -1 & +\infty & 1
\end{matrix}\right).$$
Hence
$$T = T_{2r'} \odot T_{2r'+1} =
\left(\begin{matrix}
0 & -1 & +\infty & 1 \\
0 & -1 & +\infty & 1 \\
0 & -1 & 1 & 1 \\
0 & -1 & +\infty & 1
\end{matrix}\right).$$
In this case, the sum of Equation~\eqref{eq:weakclosure} defining 
$T^+$ does not converge (this is due to the coefficient $-1$ on
the diagonal) and so, we conclude that $\zeroval(h(x)) = -\infty$.

In \S~\ref{ssec:red} we will see that all primes essentially follow the
pattern displayed for either $p=7$ or $p=11$.

\subsubsection{Complexity}

Our algorithm manipulates rational numbers whose size can
significantly grow during the execution. For this reason, we 
will estimate its complexity by counting bit operations.

In what follows,
we use the standard soft-$O$ notation $\softO(-)$ to hide
logarithmic factors. We also assume that FFT-based algorithms
are used to perform multiplications on integers; an operation
on two integers of $B$ bits then requires at most $\softO(B)$
bit operations.

\begin{proposition}
Let $r_0$ be defined as in \cref{lem:Trperiodic} and let 
$D$ be a common denominator of $\nu$ and $\nu_0$.
Then, the algorithm described in 
\S\S \ref{sssec:recurrence}--\ref{sssec:halting} performs at most
\begin{itemize}
\item Case $\nu = \nu_0$: \\
$\softO\big((n{+}m)^2 r_0 \log p + (n{+}m)^3 \log (r_0 p)\big)$ \smallskip 
\item Case $\nu > \nu_0$: \\
$\softO\big((n{+}m)^2 (r_0 \log p + \log D) \big(r_0 \log p - \log \min(1, \nu {-}\nu_0)\big)\big)$
\end{itemize}
bit operations.
\end{proposition}

\begin{proof}
For simplicity, we set $w = n + m$.

We first observe that, throughout the algorithm, the computer only 
manipulates rational numbers in $\frac 1 D \Z$.
Besides, the finite entries of the matrix $T_r$ are all in
$O\big(w + (\nu - \nu_0)p^r\big)$.

To start with, we assume that $\nu = \nu_0$. In this case, 
$D$ is a divisor of $p{-}1$, so $D \leq p$. We need
to compute $\um_r$ until $r = r_0 + e = O(r_0)$.
For a fixed $r$, computing $\um_r$ from $\um_{r-1}$ requires 
$O(s^2) \subset O(w^2)$ operations on rational numbers of order 
of magnitude $O(wr)$.
This can be done within $\softO(w^2 \log (rp))$ bit operations. In
total, computing all the requires $\mu_r$ then costs
$\softO(w^2 r_0 \log p)$ bit operations.
Finally, applying the Floyd--Warshall requires $O(s^3)$ additional
operations on rational numbers in $O(w r_0)$; this costs at most
$\softO(w^3 \log(r_0 p))$ bit operations.

We now move to the case where $\nu > \nu_0$.
Let $R$ be the first integer for which the requirements of
\cref{lem:halt} are fulfilled.
It follows from the proof of \cref{lem:halt} that $\xi_{r,1}
\geq p^{r - r_0}$ for all $r$. 
Following the proof of Proposition~\ref{prop:padicradius}, we deduce that
$$\mu_{r,1} \geq (\nu - \nu_0) p^{r-r_0} - w r + O(p^{r_0})$$
which in turn implies that
$p^R = \softO\Big(\frac{w p^{r_0}}{\nu - \nu_0} + p^{2 r_0}\Big)$.

Hence the finite entries of the matrices $T_r$ with $r \leq R$
are all in $\softO(w p^{r_0} + p^{2 r_0})$. As in the first part of the 
proof, we deduce that the computations of $\um_1, \ldots, \um_R$ 
requires at most
$$\softO\big(w^2 R (r_0 \log p + \log D)\big)$$
bit operations. We conclude the proof by noticing that

\medskip

\hfill$\textstyle R = O\big(\log w + r_0 \log p - \log \min(1, \nu{-}\nu_0)\big)$.\hfill
\end{proof}

\subsection{Newton polygons}
\label{ssec:NP}

The \emph{Newton polygon} $\NP(h)$ of the series $h(x)$ is, by 
definition, the topological closure of the convex hull in $\R^2$ 
of the points of coordinates $(k, v)$ with $v \geq \val_p(h_k)$.
The Newton polygon is closely related to the drifted valuations;
indeed, for $\nu \geq \nu_0$, we have
$$\nuval(h(x)) = \min_{(a, b) \in \NP(h)} a\nu + b.$$
To model these relations, we introduce two new tropical semirings:
\begin{enumerate}[leftmargin=1.5em]
\item the set $\mathcal F$ of concave functions $f : [\nu_0, +\infty) \to
\R \sqcup \{+\infty\}$; we endow $\mathcal F$ with the pointwise
operations $\oplus = \min$ and $\odot = +$
\item the set $\mathcal N$ of closed convex subsets $N \subset \R^2$ 
which are stable under the translations by $(0,1)$ and $(1, \nu_0)$,
endowed with
\begin{align*}
N_1 \oplus N_2 & = \textrm{convex hull of } N_1 \cup N_2, \\
N_1 \odot N_2 & = N_1 + N_2 \textrm{ (Minkowski sum).}
\end{align*}
\end{enumerate}
We have a map $\ev : \mathcal N \to \mathcal F$ that takes $N$ to
the function 
$$\textstyle \ev(N) : \nu \mapsto \min_{(a, b) \in N} a\nu + b.$$
One checks that $\ev$ is a morphism of semirings. Moreover, the
theorem of Hahn--Banach implies that it is an isomorphism, its
inverse being given by the map that takes $f \in \mathcal F$ to
the subset of $\mathbb R^2$ consisting of pairs $(a,b)$ such that
$f(\nu) \leq a \nu + b$.
A decisive advantage of $\mathcal N$ (over $\mathcal F$) is that
it provides concrete algorithmic perspectives, given that handling
operations in $\mathcal N$ is effectively doable (at least when 
the convex sets are described by finite amounts of data).

\subsubsection{Recurrence over $\mathcal N$}

We are now in position to design a uniform version (with
respect to $\nu$) of the algorithm of Subsection~\ref{ssec:val};
roughly speaking, it simply consists in replacing $\mathcal T$ by 
$\mathcal N$ everywhere. Precisely, we define the $\mu_{r,i} \in \mathcal N$
by the relations
\begin{align*}
\mu_{1,i} 
 & = \textstyle \bigoplus_{\xi \in I_{1,i}} \ev^{-1}\big(\nu \mapsto \xi\nu + w_1(\xi)\big), \\
\mu_{r,i} 
 & = \textstyle \ev^{-1}\big(\nu \mapsto w_r(\xi_{r,i})\big) \odot \Big(\bigoplus_{j \in J_{r,i}} \mu_{r-1,j}\Big).
\end{align*}
We notice in addition that $\ev^{-1}$ of the affine function $\nu 
\mapsto a \nu + b$ is simply the convex subset of $\mathbb R^2$ with 
one vertex at $(a, b)$ and two rays in the directions $(0,1)$ and
$(1,\nu_0)$.
In particular, multiplying by $\ev^{-1}(\nu \mapsto w_r(\xi_{r,i}))$ is 
nothing but translating the convex set by the vector $(0, w_r(\xi_{r,i}))$.

As in Subsection~\ref{ssec:val}, the $\mu_{r,i}$ are subject to 
a recurrence relation with respect to $i$, which reads
$$\mu_{r,i+s} = \mu_{r,i} \odot 
  \varepsilon^{-1}\big(\nu \mapsto (\nu - \nu_0) p^r + \nu_0\big)$$
(compare with Equation~\eqref{eq:muperiodicity}).
This allows to only retain the $s$ first terms of the
sequences $(\mu_{r,i})_{i \geq 0}$: setting
$\um_r = (\mu_{r,0}, \ldots, \mu_{r,s{-}1})$, we have a relation of 
the form $\um_r = \um_{r-1} \odot T_r$ where $T_r$ is now a square 
matrix of size $s$ with entries in $\mathcal N$.

Finally, the Newton polygon of $h(x)$ is obtained as the first
coordinate of
\begin{equation}
\label{eq:NP}
\textstyle
  \bigoplus_{\ell \geq 1} \um_1 \odot T_2 \odot \cdots \odot T_\ell.
\end{equation}
When $\val_{p,\mu_0}(h(x))$ is finite, it follows from what we
did in Subsection~\ref{ssec:val} and \cite[Proposition~1.6.15]{But10}
that the above infinite sum can be truncated to $\ell \leq r_0 + s$
(with the notation introduced at the end of Subsection~\ref{ssec:val})
without changing the final result. We then get a complete algorithm
to compute $\NP(h)$.

On the contrary, when $\val_{p,\mu_0}(h(x)) = -\infty$, the sum
of Equation~\eqref{eq:NP} converges but it cannot be reduced to a
finite sum: each additional term provides more and more accurate
approximations of $\NP(h)$.
This is perfectly in line with the fact that $\NP(h)$
has an infinite number of slopes in this case.
An option to nevertheless obtain a meaningful result
is to shrink a bit the domain of definition of the functions and
replace $\nu_0$ by another bound $\nu_1 > \nu_0$. Doing this, the
sum of Equation~\eqref{eq:NP} again reduces to a finite sum, and we
can decide when the computation can be stopped using the condition
of \cref{lem:halt}; indeed, inequalities of functions correspond to
inclusions of Newton polygons and so, they can easily be checked.

\subsubsection{An example}

We continue the example of Subsection~\ref{sssec:example}.
We start with $p = 7$.
In this case, finding the Newton polygon is easy. Indeed, we 
already know that $\zeroval(h(x)) = 0$. Besides, we derive from
Proposition~\ref{prop:padicradius} that $\val_p(h_k) = o(k)$.
This readily implies $\NP(h) = \R^+ \times \R^+$.

Let us nonetheless apply our algorithm and show that it outputs
the same result. In our case, the vector $\um_1$ and the matrix
$T_r \in M_4(\mathcal N)$ are the following ones:
$$\begin{array}{r@{\hspace{0.5em}}l}
\um_1 = & 
\left(
\raisebox{-0.5\height+0.2em}{%
\begin{tikzpicture}
\draw[transparent] (-0.5, 0)--(3.5, 0);
\begin{scope}[gray]
\draw (0.5, -0.5)--(0.5, 0.5);
\draw (1.5, -0.5)--(1.5, 0.5);
\draw (2.5, -0.5)--(2.5, 0.5);
\end{scope}

\begin{scope}[xshift=-0.2cm, yshift=-0.2cm, scale=0.06]
\fill[black!20] (0,0) rectangle (10,10);
\draw[thick] (10,0)--(0,0)--(0,10);
\node[below, scale=0.5] at (-1,0) { $(0,0)$ };
\end{scope}
\begin{scope}[xshift=0.8cm, yshift=-0.2cm, scale=0.06]
\fill[black!20] (2,1) rectangle (10,10);
\draw[thick] (10,1)--(2,1)--(2,10);
\node[below, scale=0.5] at (2,1) { $(2,1)$ };
\end{scope}
\begin{scope}[xshift=1.8cm, yshift=-0.2cm, scale=0.06]
\fill[black!20] (3,2) rectangle (10,10);
\draw[thick] (10,2)--(3,2)--(3,10);
\node[below, scale=0.5] at (2,2) { $(3,2)$ };
\end{scope}
\begin{scope}[xshift=2.8cm, yshift=-0.2cm, scale=0.06]
\fill[black!20] (5,1) rectangle (10,10);
\draw[thick] (10,1)--(5,1)--(5,10);
\node[below, scale=0.5] at (4,1) { $(5,1)$ };
\end{scope}
\end{tikzpicture}} \right),
\medskip \\
T_r = & \left(
\raisebox{-0.5\height+0.2em}{%
\begin{tikzpicture}
\begin{scope}[gray]
\draw (-0.5, 0.5)--(3.5, 0.5);
\draw (-0.5, 1.5)--(3.5, 1.5);
\draw (-0.5, 2.5)--(3.5, 2.5);
\draw (0.5, -0.5)--(0.5, 3.5);
\draw (1.5, -0.5)--(1.5, 3.5);
\draw (2.5, -0.5)--(2.5, 3.5);
\end{scope}

\begin{scope}[xshift=-0.2cm, yshift=2.8cm, scale=0.06]
\fill[black!20] (0,0) rectangle (10,10);
\draw[thick] (10,0)--(0,0)--(0,10);
\node[below, scale=0.5] at (-1,0) { $(0,0)$ };
\end{scope}
\begin{scope}[xshift=-0.2cm, yshift=1.8cm, scale=0.06]
\fill[black!20] (0,0) rectangle (10,10);
\draw[thick] (10,0)--(0,0)--(0,10);
\node[below, scale=0.5] at (-1,0.2) { $(0,0)$ };
\end{scope}
\begin{scope}[xshift=-0.2cm, yshift=0.8cm, scale=0.06]
\fill[black!20] (0,0) rectangle (10,10);
\draw[thick] (10,0)--(0,0)--(0,10);
\node[below, scale=0.5] at (-1,0.2) { $(0,0)$ };
\end{scope}
\begin{scope}[xshift=-0.2cm, yshift=-0.2cm, scale=0.06]
\fill[black!20] (0,0) rectangle (10,10);
\draw[thick] (10,0)--(0,0)--(0,10);
\node[below, scale=0.5] at (-1,0.2) { $(0,0)$ };
\end{scope}

\begin{scope}[xshift=0.8cm, yshift=1.8cm, scale=0.06]
\fill[black!20] (2,1) rectangle (10,10);
\draw[thick] (10,1)--(2,1)--(2,10);
\node[below, scale=0.5] at (1,1) { $(2p^{r-1},1)$ };
\end{scope}

\begin{scope}[xshift=1.8cm, yshift=2.8cm, scale=0.06]
\fill[black!20] (3,2) rectangle (10,10);
\draw[thick] (10,2)--(3,2)--(3,10);
\node[below, scale=0.5] at (2,2) { $(3p^{r-1},2)$ };
\end{scope}
\begin{scope}[xshift=1.8cm, yshift=1.8cm, scale=0.06]
\fill[black!20] (3,2) rectangle (10,10);
\draw[thick] (10,2)--(3,2)--(3,10);
\node[below, scale=0.5] at (2,2) { $(3p^{r-1},2)$ };
\end{scope}
\begin{scope}[xshift=1.8cm, yshift=0.8cm, scale=0.06]
\fill[black!20] (2,2) rectangle (10,10);
\draw[thick] (10,2)--(2,2)--(2,10);
\node[below, scale=0.5] at (1,2) { $(2p^{r-1},2)$ };
\end{scope}
\begin{scope}[xshift=1.8cm, yshift=-0.2cm, scale=0.06]
\fill[black!20] (2,2) rectangle (10,10);
\draw[thick] (10,2)--(2,2)--(2,10);
\node[below, scale=0.5] at (1,2) { $(2p^{r-1},2)$ };
\end{scope}

\begin{scope}[xshift=2.8cm, yshift=2.8cm, scale=0.06]
\fill[black!20] (5,1) rectangle (10,10);
\draw[thick] (10,1)--(5,1)--(5,10);
\node[below, scale=0.5] at (4,1) { $(5p^{r-1},1)$ };
\end{scope}
\begin{scope}[xshift=2.8cm, yshift=1.8cm, scale=0.06]
\fill[black!20] (5,1) rectangle (10,10);
\draw[thick] (10,1)--(5,1)--(5,10);
\node[below, scale=0.5] at (4,1) { $(5p^{r-1},1)$ };
\end{scope}
\begin{scope}[xshift=2.8cm, yshift=0.8cm, scale=0.06]
\fill[black!20] (5,1) rectangle (10,10);
\draw[thick] (10,1)--(5,1)--(5,10);
\node[below, scale=0.5] at (4,1) { $(5p^{r-1},1)$ };
\end{scope}
\begin{scope}[xshift=2.8cm, yshift=-0.2cm, scale=0.06]
\fill[black!20] (4,1) rectangle (10,10);
\draw[thick] (10,1)--(4,1)--(4,10);
\node[below, scale=0.5] at (3,1) { $(4p^{r-1},1)$ };
\end{scope}

\end{tikzpicture}}\right).
\end{array}$$

\noindent
We recall that the Newton polygon we are looking for is the
limit of the first coordinate of Equation~\eqref{eq:NP} when
$r$ goes to infinity. In our case, we observe that all the
entries of $\um_1$ and $T_r$ are subsets of $\R^+ \times \R^+$.
Therefore, for all $r \geq 2$, the four coordinates of $\um_1
\odot T_2 \odot \cdots \odot T_r$ are also subsets of
$\R^+ \times \R^+$. We conclude that $\NP(h) = \R^+ \times
\R^+$ as expected.

When $p = 11$, on the contrary, we are in the situation where
$\val_{p,0}(h(x)) = -\infty$, meaning that the limit of
Equation~\eqref{eq:NP} is not reached at finite level. In this
case, the matrices $T_r$ are those shown on Figure~\ref{fig:Tr} 
and we can see that the entries of the second column of $T_r$ 
are not subsets of $\R^+ \times \R^+$ when $r$ is odd.
What happens more precisely is that each summand in
Equation~\eqref{eq:NP} corresponding to an odd number $\ell = 
2r'-1$ introduces a new vertex $V_{r'}$ with $y$-coordinate equal 
to $-r'$. This vertex then propagates to the first coordinate 
and eventually contributes to the Newton polygon.

As explained in the description of the algorithm, we can avoid
this by slightly degrading the precision and adding a ray in 
the direction $(0, \nu_1)$ (with $\nu_1 > 0$ fixed) to all
entries of $\um_1$ and $T_r$.
This will absorb the vertices $V_{r'}$ for $r'$ large enough.

\begin{remark}
Carrying out all computations, one finds that the vertices of
$\NP(h)$ are the points
$\Big(4 \frac{p^{2r'} - 1}{p^2 - 1}, -r'\Big)$ with $r' \in \N$.
In this particular example, they then exhibit very strong
regularity patterns. One may wonder whether this is a general 
phenomenon and, if it is, if we can use it to compute $\NP(h)$ 
and/or to accelerate the computation of $\nuval(h(x))$ when 
$\nu > \nu_0$.
\end{remark}

\begin{figure*}
\centering
$T_r = \left(
\raisebox{-0.5\height+0.2em}{%
\begin{tikzpicture}
\begin{scope}[gray]
\draw (-0.5, 0.5)--(3.5, 0.5);
\draw (-0.5, 1.5)--(3.5, 1.5);
\draw (-0.5, 2.5)--(3.5, 2.5);
\draw (0.5, -0.5)--(0.5, 3.5);
\draw (1.5, -0.5)--(1.5, 3.5);
\draw (2.5, -0.5)--(2.5, 3.5);
\end{scope}

\begin{scope}[xshift=-0.2cm, yshift=2.8cm, scale=0.05]
\fill[black!20] (0,0) rectangle (12,12);
\draw[thick] (12,0)--(0,0)--(0,12);
\node[below, scale=0.5] at (-1,0) { $(0,0)$ };
\end{scope}
\begin{scope}[xshift=-0.2cm, yshift=1.8cm, scale=0.05]
\fill[black!20] (0,0) rectangle (12,12);
\draw[thick] (12,0)--(0,0)--(0,12);
\node[below, scale=0.5] at (-1,0.2) { $(0,0)$ };
\end{scope}
\begin{scope}[xshift=-0.2cm, yshift=0.8cm, scale=0.05]
\fill[black!20] (0,0) rectangle (12,12);
\draw[thick] (12,0)--(0,0)--(0,12);
\node[below, scale=0.5] at (-1,0.2) { $(0,0)$ };
\end{scope}
\begin{scope}[xshift=-0.2cm, yshift=-0.2cm, scale=0.05]
\fill[black!20] (0,0) rectangle (12,12);
\draw[thick] (12,0)--(0,0)--(0,12);
\node[below, scale=0.5] at (-1,0.2) { $(0,0)$ };
\end{scope}

\begin{scope}[xshift=0.8cm, yshift=0.8cm, scale=0.05]
\fill[black!20] (3,1) rectangle (12,12);
\draw[thick] (12,1)--(3,1)--(3,12);
\node[below, scale=0.5] at (2,1) { $(3p^{r-1},1)$ };
\end{scope}

\begin{scope}[xshift=1.8cm, yshift=2.8cm, scale=0.05]
\fill[black!20] (4,2) rectangle (12,12);
\draw[thick] (12,2)--(4,2)--(4,12);
\node[below, scale=0.5] at (3,2) { $(4p^{r-1},2)$ };
\end{scope}
\begin{scope}[xshift=1.8cm, yshift=1.8cm, scale=0.05]
\fill[black!20] (4,2) rectangle (12,12);
\draw[thick] (12,2)--(4,2)--(4,12);
\node[below, scale=0.5] at (3,2) { $(4p^{r-1},2)$ };
\end{scope}
\begin{scope}[xshift=1.8cm, yshift=0.8cm, scale=0.05]
\fill[black!20] (4,2) rectangle (12,12);
\draw[thick] (12,2)--(4,2)--(4,12);
\node[below, scale=0.5] at (3,2) { $(4p^{r-1},2)$ };
\end{scope}
\begin{scope}[xshift=1.8cm, yshift=-0.2cm, scale=0.05]
\fill[black!20] (3,2) rectangle (12,12);
\draw[thick] (12,2)--(3,2)--(3,12);
\node[below, scale=0.5] at (2,2) { $(3p^{r-1},2)$ };
\end{scope}

\begin{scope}[xshift=2.8cm, yshift=2.8cm, scale=0.05]
\fill[black!20] (8,1) rectangle (12,12);
\draw[thick] (12,1)--(8,1)--(8,12);
\node[below, scale=0.5] at (7,1) { $(8p^{r-1},1)$ };
\end{scope}
\begin{scope}[xshift=2.8cm, yshift=1.8cm, scale=0.05]
\fill[black!20] (7,1) rectangle (12,12);
\draw[thick] (12,1)--(7,1)--(7,12);
\node[below, scale=0.5] at (6,1) { $(7p^{r-1},1)$ };
\end{scope}
\begin{scope}[xshift=2.8cm, yshift=0.8cm, scale=0.05]
\fill[black!20] (7,1) rectangle (12,12);
\draw[thick] (12,1)--(7,1)--(7,12);
\node[below, scale=0.5] at (6,1) { $(7p^{r-1},1)$ };
\end{scope}
\begin{scope}[xshift=2.8cm, yshift=-0.2cm, scale=0.05]
\fill[black!20] (7,1) rectangle (12,12);
\draw[thick] (12,1)--(7,1)--(7,12);
\node[below, scale=0.5] at (6,1) { $(7p^{r-1},1)$ };
\end{scope}

\end{tikzpicture}}\right)$
\quad ($r$ even)
\qquad ; \qquad
$T_r = \left(
\raisebox{-0.5\height+0.2em}{%
\begin{tikzpicture}
\begin{scope}[gray]
\draw (-0.5, 0.5)--(3.5, 0.5);
\draw (-0.5, 1.5)--(3.5, 1.5);
\draw (-0.5, 2.5)--(3.5, 2.5);
\draw (0.5, -0.5)--(0.5, 3.5);
\draw (1.5, -0.5)--(1.5, 3.5);
\draw (2.5, -0.5)--(2.5, 3.5);
\end{scope}

\begin{scope}[xshift=-0.2cm, yshift=2.8cm, scale=0.05]
\fill[black!20] (0,0) rectangle (12,12);
\draw[thick] (12,0)--(0,0)--(0,12);
\node[below, scale=0.5] at (-1,0) { $(0,0)$ };
\end{scope}
\begin{scope}[xshift=-0.2cm, yshift=1.8cm, scale=0.05]
\fill[black!20] (0,0) rectangle (12,12);
\draw[thick] (12,0)--(0,0)--(0,12);
\node[below, scale=0.5] at (-1,0.2) { $(0,0)$ };
\end{scope}
\begin{scope}[xshift=-0.2cm, yshift=0.8cm, scale=0.05]
\fill[black!20] (0,0) rectangle (12,12);
\draw[thick] (12,0)--(0,0)--(0,12);
\node[below, scale=0.5] at (-1,0.2) { $(0,0)$ };
\end{scope}
\begin{scope}[xshift=-0.2cm, yshift=-0.2cm, scale=0.05]
\fill[black!20] (0,0) rectangle (12,12);
\draw[thick] (12,0)--(0,0)--(0,12);
\node[below, scale=0.5] at (-1,0.2) { $(0,0)$ };
\end{scope}

\begin{scope}[xshift=0.8cm, yshift=2.8cm, scale=0.05]
\fill[black!20] (4,-1) rectangle (12,12);
\draw[thick] (12,-1)--(4,-1)--(4,12);
\node[below, scale=0.5] at (3,-1) { $(4p^{r-1},-1)$ };
\end{scope}
\begin{scope}[xshift=0.8cm, yshift=1.8cm, scale=0.05]
\fill[black!20] (4,-1) rectangle (12,12);
\draw[thick] (12,-1)--(4,-1)--(4,12);
\node[below, scale=0.5] at (3,-1) { $(4p^{r-1},-1)$ };
\end{scope}
\begin{scope}[xshift=0.8cm, yshift=0.8cm, scale=0.05]
\fill[black!20] (4,-1) rectangle (12,12);
\draw[thick] (12,-1)--(4,-1)--(4,12);
\node[below, scale=0.5] at (3,-1) { $(4p^{r-1},-1)$ };
\end{scope}
\begin{scope}[xshift=0.8cm, yshift=-0.2cm, scale=0.05]
\fill[black!20] (3,-1) rectangle (12,12);
\draw[thick] (12,-1)--(3,-1)--(3,12);
\node[below, scale=0.5] at (2,-1) { $(3p^{r-1},-1)$ };
\end{scope}

\begin{scope}[xshift=1.8cm, yshift=1.8cm, scale=0.05]
\fill[black!20] (7,0) rectangle (12,12);
\draw[thick] (12,0)--(7,0)--(7,12);
\node[below, scale=0.5] at (6,0) { $(7p^{r-1},0)$ };
\end{scope}

\begin{scope}[xshift=2.8cm, yshift=2.8cm, scale=0.05]
\fill[black!20] (8,1) rectangle (12,12);
\draw[thick] (12,1)--(8,1)--(8,12);
\node[below, scale=0.5] at (7,1) { $(8p^{r-1},1)$ };
\end{scope}
\begin{scope}[xshift=2.8cm, yshift=1.8cm, scale=0.05]
\fill[black!20] (8,1) rectangle (12,12);
\draw[thick] (12,1)--(8,1)--(8,12);
\node[below, scale=0.5] at (7,1) { $(8p^{r-1},1)$ };
\end{scope}
\begin{scope}[xshift=2.8cm, yshift=0.8cm, scale=0.05]
\fill[black!20] (7,1) rectangle (12,12);
\draw[thick] (12,1)--(7,1)--(7,12);
\node[below, scale=0.5] at (6,1) { $(7p^{r-1},1)$ };
\end{scope}
\begin{scope}[xshift=2.8cm, yshift=-0.2cm, scale=0.05]
\fill[black!20] (7,1) rectangle (12,12);
\draw[thick] (12,1)--(7,1)--(7,12);
\node[below, scale=0.5] at (6,1) { $(7p^{r-1},1)$ };
\end{scope}

\end{tikzpicture}}\right)$
\quad ($r$ odd)

\caption{\rm The matrices $T_r$ for the hypergeometric series $\pFq{\big(\frac 1 3, \frac 4 3\big)}{\big(\frac 2 3\big)}{x}$ and $p = 11$}
\label{fig:Tr}
\end{figure*}

\subsection{Application to $p$-adic evaluation}
\label{ssec:padiceval}

Proposition~\ref{prop:padicradius} ensures that, if $\val_p(x) > 
\nu_0$, then $\val_p(h_k x^k)$ goes to infinity when $k$ grows, 
and so the hypergeometric series $\sum_k h_k x^k$ converges for
the $p$-adic topology. Hence, it defines a function
$h : B_{\nu_0} \to \Q_p$
where $\Q_p$ is the field of $p$-adic numbers and $B_{\nu_0}$ is
its open disc centered at $0$ of radius $p^{-\nu_0}$. In what
follows, we briefly outline an algorithm to compute $h(a)$ at 
precision $O(p^N)$ for given $a \in B_{\nu_0}$ and $N \in \Z$.

\begin{lemma}
\label{lem:truncation}
Let $\nu$ be a real number the interval $(\nu_0, \val_p(a))$.
Then $\val_k(h_k a^k) \geq N$ for all
$k \geq K \coloneq \frac{N - \nuval(h(x))}{\val_p(a) - \nu}$.
\end{lemma}

\begin{proof}
The lemma follows from the inequality
$$\val_p(h_k a^k) 
 \geq \nuval(h(x)) + k\cdot(\val_p(a) - \nu)$$
which is a direct consequence of the definition of $\nuval(h(x))$.
\end{proof}

To compute $h(a)$, we then proceed as follows:

\noindent
(1) we choose a rational number $\nu \in (\nu_0, \val_p(a))$,

\noindent
(2) we compute $\nuval(h(x))$ using the algorithm of \S\ref{ssec:val},

\noindent
(3) we compute the bound $K$ of Lemma~\ref{lem:truncation},

\noindent
(4) we output $\sum_{k<K} h_k a^k + O(p^N)$.

\medskip

Although the previous algorithm works for any value of $\nu$, bad
choices could lead to huge truncation bounds and so, dramatic
performances.
Nonetheless, we know from the proof of Proposition~\ref{prop:padicradius}
that the order of magnitude of $\nu_0 k - \val_p(h_k)$ is about
$\log_p (k)$.
Using this approximation and solving the corresponding optimization
problem, one finds the following heuristic for the choice of $\nu$:
$$\textstyle
\nu = \nu_0 + \frac{\val_p(a) - \nu_0}{N} \in (\nu_0, \val_p(a)).$$
In practice, this choice leads to $K$ close to 
$\frac N{\val_p(a) - \nu_0}$, which is basically the best we can hope
for.

\section{Reduction modulo primes}

We fix a prime number $p$ and we let
$\Z_{(p)}$ denote the subring of $\Q$ consisting of rational
numbers $x$ such that $\val_p(x) \geq 0$, \emph{i.e.} rational
numbers $\frac a b$ with $\gcd(b, p) = 1$.
Any element of $\Z_{(p)}$ can be reduced modulo $p$, yielding a
result in $\Fp \coloneq \Z/p\Z$.

It may happen for some parameters $\ua$, $\ub$ that \emph{all} the
coefficients $(h_k)_{k \geq 0}$ of $h(x) = \pFq{\ua}{\ub} x$
lie in $\Z_{(p)}$ or, equivalently, that $\zeroval(h(x)) \geq 0$.
In this case, we say that $h(x)$ has \emph{good reduction} modulo $p$
and we write $h(x) \bmod p$ for the image of $h(x)$ in $\Fp\ps x$.

%

\subsection{Good reduction primes}
\label{ssec:red}

We fix two tuples to parameters $\ua \in \Q^n$ and $\ub \in \Q^m$ 
and set
$$\textstyle
h(x) = \sum_{k \geq 0} h_k x^k \coloneq \pFq{\ua}{\ub} x.$$
Checking if $h(x)$ has good reduction at $p$ can be done using the 
algorithm of Subsection~\ref{ssec:val}: we compute $\zeroval(h(x))$ and 
look whether it is negative or not. However, describing explicitly 
the set $\mathcal P_h$ of all primes $p$ at which $h(x)$ has good 
reduction looks more challenging.
The aim of this subsection is to answer this question.

Let $d$ be the smallest common divisors of the parameters in $\ua$
and $\ub$. Since $d$ has of course only finitely many prime divisors,
it is easy to treat them separately. Hence, in what follows, we
shall always assume that $p$ does not divide $d$.

When $m > n$, we deduce from \cref{prop:padicradius}
that $\val_p(h_k)$ goes to
$-\infty$ when $k$ grows. Hence $h(x)$ cannot have good 
reduction in this case. From now on, we then assume that
$m \leq n$.

If $x$ is a real number, we denote by $\{x\}$ its decimal part,
that is, by definition, the unique real number in $[0,1)$ such
that $x - \{x\} \in \Z$.
We will need the following result, which is a slight reformulation
of~\cite[Lemma~4]{Chr86}.

\begin{lemma}[Christol]
\label{lem:christol}
Let $\gamma = \frac a d \in \Q$.
Let $q > |a{-}d|$ be an integer which is coprime with $d$ and
let $\Delta$ be the unique integer in $\{1, \ldots, d{-}1\}$
such that $\Delta q \equiv 1 \pmod d$.

Then the reduction of $1{-}\gamma$ modulo $q$ is
$(1{-}\gamma) + q{\cdot}\{\gamma \Delta\}$.
\end{lemma}

Writing $\denom(x, y)$ for the smallest common denominator of
$x$ and $y$, we set
$$B(\ua,\ub) \coloneq 
\max_{\gamma, \gamma' \in \Gamma} \,\,
  \denom(\gamma, \gamma') {\cdot} |\gamma{-}\gamma'|$$
where we recall that $\Gamma = \{1, \alpha_1, \ldots, \alpha_n,
\beta_1, \ldots, \beta_m\}$.
\cref{lem:christol} implies that, for $p > B(\ua,\ub)$ and $r \geq 1$,
the ordering of the $1{-}\gamma \bmod p^r$, for $\gamma \in \Gamma$ only
depends on the congruence of $p$ modulo a common denominator $d$ of
the $\alpha_i$ and $\beta_j$.

\paragraph{Case $m = n$}

Here, in virtue of what we did previously, the vector $\um_1$ and
the matrices $T_r$ of Subsection~\ref{ssec:val}  
only depend on $p \bmod d$ provided that $p > B(\ua,\ub)$.
This is then also the case for $\zeroval(h(x))$ and, consequently,
for the fact that $h(x)$ has or has not good reduction modulo $p$.
In other words, $\mathcal P_h$ is the union of exceptional primes
up to the bound $B(\ua,\ub)$ and then, it consists of primes in
arithmetic progressions of ratio $d$.

These observations also readily give an algorithm to compute
$\mathcal P_h$: we check whether $h(x)$ has good or bad reduction
for all primes $p \leq B(\ua,\ub)$ and for one representative of
each invertible class modulo $d$.
We mention nonetheless that checking congruence classes can be
done more efficiently by using the criterion of 
\cite[Theorem~3.1.3]{CFV25}.

\paragraph{Case $m < n$}

In this case, the matrix $\um_1$ is again independent from $p$
as soon as $p > B(\ua,\ub)$.
However, the matrices $T_r$ are not. Following their construction, 
we nonetheless realize that they can be written as 
$$\textstyle 
T_r = U_r + \frac{p^{r-1}-1}{p-1} \cdot V_r$$
where $U_r$ and $V_r$ are independent from $p$ for $p > B(\ua,\ub)$.
Besides $U_r \geq -m$ (in the sense that all its entries are at
least $-m$) and $V_r \geq 0$. In particular $T_2$ does not depend
on $p$ as well and so neither does $\um_2 = \um_1 \odot T_2$.
Define $\tilde \um_r \coloneq \um_r \oplus \big(0, \ldots, 0\big)$.
By induction on $r$, we prove that $\um_r \geq -mr$ for all $r$,
which in turns implies that
$\tilde \um_r = \tilde \um_{r-1} \odot U_r$ if
$p^{r-1}{-}1 > mr(p{-}1)$.
We conclude that $\tilde \um_r$ only depends on 
$r \bmod d$ provided that $p > B(\ua,\ub)$ as before and 
$p^{r-1}{-}1 > mr(p{-}1)$ for all $r \geq 3$.
The latter condition is fulfilled as soon as $p \geq 2m$; indeed
from $p \geq 2m$, we derive $p^i \geq 2m$ for all $i$ and then
$$\textstyle
\frac{p^{r-1}-1}{p{-}1} = 
1 + p + \cdots + p^{r-2} \geq 1 + 2m(r{-}2) > mr
\quad (r \geq 3).$$
We now recall that the first coordinate of $\um_r$ tends to
$\zeroval(h(x))$ when $r$ goes to infinity. Therefore, the limit
of the first coordinate of $\tilde \um_r$ is 
$\zeroval(h(x)) \oplus 0= \min\big(\zeroval(h(x)), 0\big)$. The fact 
that $h(x)$ has good reduction at $p$ can then be read off on the
$\tilde \um_r$, and thus only depend on the congruence class of
$p$ modulo $d$ provided that $p > \max(B(\ua,\ub), 2m)$.

As a consequence, one can compute the set $\mathcal P_f$ by 
proceeding as in the case $m=n$, except that the bound $B(\ua,\ub)$
needs now to be replaced by $\max(B(\ua,\ub), 2m)$.

\subsection{Section operators}

From now on, we fix a prime number $p$ of good reduction for $h(x)$
and aim at studying the algebraic properties of $h(x) \bmod p$.

We let $n'$ (resp. $m'$) be the number of top (resp. bottom) parameters
which are in $\Z_{(p)}$ and let $v'$ (resp. $w'$) be the sum of the
valuations of the top (resp. bottom) parameters which are not in
$\Z_{(p)}$. Following what we did in Subsection~\ref{ssec:val},
we define
$\nu_0 \coloneq w' - v' + \frac{m' - n'}{p-1}$.

From Proposition~\ref{prop:padicradius}, we derive that
$\lim_{k \to \infty} \val_p(h_k)/ k = - \nu_0$. Hence,
our good reduction assumption ensures that $\nu_0 \leq 0$.
We also set
$\nu \coloneq (1-p)\nu_0 = (p-1)(w'- v') + (m' - n') \in \N$.

Key tools for studying $h(x) \bmod p$ are the section operators
that we introduce now.

\begin{definition}
Let $r$ be a nonnegative integer. The map
\[ \begin{array}{rrl}
    \Lambda_r:& \Z_{(p)}\ps x &\to \Z_{(p)}\ps x \smallskip \\
    & \sum_{k=0}^\infty a_k x^k &\mapsto \sum_{k=0}^\infty a_{kp+r} x^k
\end{array}\]
is called the $r$-th \emph{section operator}.
\end{definition}

We are going to prove that the $r$-th section of a hypergeometric series 
is closely related to a scalar multiple of another hypergeometric series.
Before proceeding, we need the two following definitions.

\begin{definition}
For $a, b \in \Q$, we say that $a$ is \emph{multiplicatively congruent}
to $b$ modulo $p$ and we write $a \equivm b \pmod p$ if $a = b = 0$,
or $b \neq 0$ and $\frac a b \in 1 + p \Z_{(p)}$.

\noindent
Similarly, if $f(x) = \sum_k a_k x^k$ and $g(x) = \sum_k
b_k x^k$ are two series with coefficients in $\Q$, we write
$f \equivm g \pmod p$ if $a_k \equivm b_k \pmod p$ for all $k$.
\end{definition}

\begin{definition}[Dwork map]
We define the map $\Dp : \Q \to \Q$ by:
\begin{itemize}
\item when $\gamma \in \Z_{(p)}$, $\Dp(\gamma)$ is the unique element of
$\Z_{(p)}$ such that $p\Dp(\gamma) - \gamma \in \{0, \ldots, p{-}1\}$,
\item when $\gamma \not\in \Z_{(p)}$, $\Dp(\gamma) \coloneqq \gamma$.
\end{itemize}
\end{definition}

\begin{proposition}
\label{prop:section}
For $r \geq 0$, we have
$$\Lambda_r(h(x)) \equivm h_r \cdot\pFq{\Dp(\ua{+}r)}{\Dp(\ub{+}r)} {(-p)^\nu x} \pmod p.$$
\end{proposition}

\begin{proof}
A direct computation gives
$$\textstyle
h_r \cdot \pFq{\ua{+}r}{\ub{+}r} x = \sum_{k\geq 0} h_{k+r} x^k.$$
Therefore, we can assume without loss of generality that $r = 0$.
Let $\gamma \in \Z_{(p)}$. Among $\gamma, \gamma{+}1, \ldots, 
\gamma{+}p{-}1$, one finds all congruence classes modulo
$p$. Besides, the unique $\gamma + a$ (for $0 \leq a < p$) which lies
in $p \Z_{(p)}$ is $p \Dp(\gamma)$ by definition of the Dwork map.
Thus we get the multiplicative congruence
$$(\gamma)_p = \gamma (\gamma{+}1) \cdots (\gamma{+}p{-}1)
  \equivm (p{-}1)! \cdot p \Dp(\gamma) \pmod p.$$
Hence $(\gamma)_p \equivm -p\Dp(\gamma) \pmod p$ using Wilson's theorem.
Repeating the argument $k$ times and using $\Dp(\gamma + p) = \Dp(\gamma) + 1$,
we end up with
$$(\gamma)_{pk} \equivm (-p)^k \cdot (\Dp(\gamma))_k \pmod p.$$
If now $\gamma \not\in \Z_{(p)}$, we find instead, using Fermat's theorem,
\begin{align*}
(\gamma)_{pk} \equivm \gamma^{pk} 
& \equivm p^{k(p-1)\val_p(\gamma)} \cdot \gamma^k \\
& \equivm (-p)^{k(p-1)\val_p(\gamma)} \cdot (\Dp(\gamma))_k \pmod p.
\end{align*}
For the last congruence, we used $(-1)^{p-1} \equiv 1 \pmod p$.

Putting now all together, we obtain
$$h_{pk} 
  \equivm (-p)^{k\nu} \cdot 
          \frac{(\Dp(\alpha_1))_k \cdots (\Dp(\alpha_n))_k}
               {(\Dp(\beta_1))_k \cdots (\Dp(\beta_m))_k} \pmod p$$
which gives the desired multiplicative congruence.
\end{proof}

\begin{remark}
\cref{prop:section} yields an algorithm with complexity 
$O\big((m{+}n)p \log N\big)$
for computing the multiplicative congruence class modulo $p$ of $h_N$.
Indeed, if $N = p N_1 + r_0$ is the Euclidean division of $N$
by $p$, we have $h_N \equivm h_{r_0} h^{(1)}_{N_1} \pmod p$
where
$$\textstyle
h^{(1)}(x) = \sum_{k \geq 0} h^{(1)}_k x^k 
  \coloneq \pFq{\Dp(\ua{+}r)}{\Dp(\ub{+}r)} {(-p)^\nu x}.$$
Repeating this argument again and again, we finally end up with
a multiplicative congruence of the form
$$h_N \equivm h_{r_0} h^{(1)}_{r_1} h^{(2)}_{r_2}
\cdots h^{(\ell)}_{r_\ell} \pmod p$$
where $\ell$ is the number of digits of $N$ in base $p$, hence 
$\ell \in O(\log N)$. Since moreover $r_i < p$ for all $i$, each
term $h^{(i)}_{r_i}$ can be computed for a cost of $O\big((m{+}n)p\big)$ 
multiplications and divisions.
\end{remark}

\cref{prop:section} can be rephrased by means of classical
congruences as follows.

\begin{corollary}
\label{cor:section}
Let $r$ be a nonnegative integer. Set
$$\textstyle
  g(x) = \sum_{k\geq 0} g_k x^k \coloneq \pFq{\Dp(\ua{+}r)}{\Dp(\ub{+}r)} x.$$
Then 
$\val_p(h_r) + \val_\nu(g(x)) \geq 0$ and
\begin{itemize}
\item
if $\val_p(h_r) + \val_\nu(g(x)) > 0$, then
$\Lambda_r(h(x)) \equiv 0 \pmod p$, \smallskip
\item
if $\val_p(h_r) + \val_\nu(g(x)) = 0$, then
$$\begin{array}{l}
\Lambda_r(h(x)) \\
\hspace{1mm}%
\equiv h_{r+ps} x^s \cdot \pFq{\Dp(\ua{+}r){+}s}{\Dp(\ub{+}r){+}s} {(-p)^\nu x} \!\pmod p
\end{array}$$
\end{itemize}
where $s$ is the smallest integer such that $\val_p(g_s) + \nu s = \val_\nu(g(x))$.
\end{corollary}

We emphasize that Corollary~\ref{cor:section} is effective in
the sense that $\val_\nu(g(x))$ and $s$ can both be computed using 
the algorithm of Subsection~\ref{ssec:val} (see also \cref{rem:smallest}).

\subsection{Algebraicity modulo $p$}
\label{ssec:algebraicity}


Writing 
$h(x) \equiv \sum_{r=0}^{p-1} x^r \cdot \Lambda_r(h(x))^p \pmod p$,
we derive from \cref{cor:section} that $h(x) \bmod p$ can be written as 
a linear
combination over $\Fp[x]$ of $p$-th powers of other hypergeometric series.
We call this identity the \emph{Dwork relation} associated to $h(x)$.
Our aim in this subsection is to derive from Dwork relations an algebraic
relation involving uniquely $h(x) \bmod p$.

When $\nu_0 < 0$, the series $h(x) \bmod p$ is a polynomial, and algebraicity
is clear. From now on, we then assume that $\nu_0 = 0$; hence $\nu = 0$ as
well.

The main ingredient of our proof is the fact that iterating Dwork
relations, we only encounter a finite number of auxiliary hypergeometric
series. To establish this, we define $X$ as the smallest subset of $\Q^n 
\times \Q^m$ containing $(\ua, \ub)$ and stable by the operations 
$$(\ua', \ub') \mapsto \big(\Dp(\ua'{+}r), \Dp(\ub'{+}r)\big) 
  \quad (0 \leq r < p).$$
Using that $\Dp(\gamma) = \frac \gamma p + O(1)$ for $\gamma \in \Z_{(p)}$,
we deduce that $X$ is finite (see also \cite[Lemma~3.2.1~(2)]{CFV25}).
Let $(\ua', \ub') \in X$ and write
$$\textstyle
  g(x) = \sum_{k\geq 0} g_k x^k \coloneq \pFq{\ua'}{\ub'} x.$$
If $\zeroval(g(x)) > -\infty$, we let $t$ denote the smallest
integer for which $\val_p(g_t) = \zeroval(g(x))$ and consider the 
new parameters $(\ua'{+}t, \ub'{+}t)$.
We let $Y$ denote the set of all parameters obtained this way by letting
$(\ua', \ub')$ vary in $X$. Clearly, $Y$ is finite as well.

\begin{proposition}
\label{prop:Dworkiterated}
For any nonnegative integer $e$, there exists a relation of the form
$$h(x) \equiv P(x) + \sum_{\ug \in Y} Q_{\ug}(x) \cdot \pGq{\ug} x^{p^e}
\pmod p$$
where $P(x)$ and $Q_{\ug}(x)$ are polynomials in $\Fp[x]$.
\end{proposition}

\begin{proof}
Let $\ug \in Y$.
Then $\ug = (\ua'{+}s, \ub'{+}s)$ for some $(\ua',\ub') \in X$ and
$s \in \N$.
The Dwork relation makes $\pGq{\ug} x \bmod p$ appear as a
$\Fp[x]$-linear combination of the series
$$\pFq{\Dp(\ua'{+}s{+}r)+s_r}{\Dp(\ub'{+}s{+}r)+s_r} x^p \bmod p
  \quad (0 \leq r < p)$$
for some $s_r \in \N$.
We observe that, if $s{+}r = ap + b$ is the Euclidean division of
$s{+}r$ by $p$, we have $\Dp(\gamma{+}s{+}r) = \Dp(\gamma{+}b) + a$
for all $\gamma \in \Z_{(p)}$. Moreover, when $\gamma \not\in \Z_{(p)}$,
shifting $\gamma$ by any integer does not change the hypergeometric
series modulo $p$. Therefore, $\pGq{\ug} x \bmod p$
is also in the $\Fp[x]$-linear span of the
$$\pFq{\Dp(\ua'{+}b)+t_b}{\Dp(\ub'{+}b)+t_b} x^p \bmod p
  \quad (0 \leq b < p)$$
for some $t_b \in \N$.
Moreover, for each fixed $b \in \{0, \ldots, p{-}1\}$, the pair
$(\ua'', \ub'') \coloneq (\Dp(\ua'{+}b), \Dp(\ub'{+}b))$ is by
definition in $X$. Let $g(x) = \sum_k g_k x^k$ be the associated 
hypergeometric series and let $t$ be the smallest integer such that
$\val_p(g_t) = \zeroval(g(x))$
which $\pFq{\ua''{+}t}{\ub''{+}t} x$ has good reduction
modulo $p$. Thus $t_b \geq t$ and there exists $c_b \in \Fp$, 
$C_b(x) \in \Fp[x]$ such that
$$\begin{array}{l}
\pFq{\ua''{+}t_b}{\ub''{+}t_b} x \smallskip \\
\hspace{1cm}
\equiv C_b(x) + c_b x^{t_b - t} \pFq{\ua''{+}t}{\ub''{+}t} x
\pmod p.
\end{array}$$
It follows that $\pGq{\ug} x \bmod p$ is in the $\Fp[x]$-linear
span of $1$ and the hypergeometric series $\pGq{\ug'} x^p \bmod p$
for $\ug'$ varying in $Y$.

The proposition follows by induction on $e$.
\end{proof}

\begin{theorem}
\label{th:algebraic}
The series $h(x) \bmod p$ is algebraic over $\Fp(x)$.
\end{theorem}

\begin{proof}
Set $N \coloneq \text{Card}(Y) + 1$.
Raising the relation of \cref{prop:Dworkiterated} to the power
$p^{N-e}$, we obtain 
$$\textstyle
h(x)^{p^{N-e}} 
  \equiv \sum_{\ug \in Y^\bullet} Q_{\ug}(x)^{p^{N-e}} \cdot \pGq{\ug} x^{p^N} \pmod p$$
where we have set $Y^\bullet = Y \sqcup \{\bullet\}$ and $Q_{\bullet}(x) = 
P(x)$, $\pGq{\bullet} x = 1$.
These equalities give a linear system, yielding a single 
matrix relation of the form
$$\Big(h(x)^{p^e}\Big)_{0 \leq e \leq N} 
\equiv \Big( \pGq{\ug} x^{p^N} \Big)_{\ug \in Y^\bullet} \cdot M(x) \pmod p$$
where $M(x)$ is a matrix over $\Fp[x]$ with rows indexed by $Y^\bullet$
and columns indexed by $\{0, \ldots, N\}$.
Hence $M(x)$ has more columns than rows,
and so it has a nontrivial vector $(v_0(x), \ldots, v_N(x))$
in its right kernel. Thus
$v_0(x) h(x) + v_1(x) h(x)^p + \cdots + v_N(x) h(x)^{p^N} = 0$,
proving algebraicity.
\end{proof}

Again, we underline that the proofs of \cref{prop:Dworkiterated}
and \cref{th:algebraic} readily yield an algorithm for computing an
annihilating polynomial of $h(x) \bmod p$.
Its complexity is polynomial when we count operations in $\Fp[x]$
but the size of the polynomials can grow very rapidly, due to
frequent raisings to the $p$-th power.
We believe nonetheless that this blow up is intrinsic to the
problem in the sense that a general hypergeometric series $h(x)
\bmod p$ has only huge annihilating polynomials with coefficients
having degree growing exponentially fast with respect to~$p$.

\newpage
\printbibliography

@article{AKM20,
  title = {On {{Christol}}’s Conjecture},
  author = {Abdelaziz, Y and Koutschan, C and Maillard, J-M},
  date = {2020-05-22},
  journaltitle = {Journal of Physics A: Mathematical and Theoretical},
  shortjournal = {J. Phys. A: Math. Theor.},
  volume = {53},
  number = {20},
  pages = {205201},
  doi = {10.1088/1751-8121/ab82dc}
}

@online{BBC+12,
  title = {Ising N-Fold Integrals as Diagonals of Rational Functions and Integrality of Series Expansions: Integrality versus Modularity},
  shorttitle = {Ising N-Fold Integrals as Diagonals of Rational Functions and Integrality of Series Expansions},
  author = {Bostan, A. and Boukraa, S. and Christol, G. and Hassani, S. and Maillard, J-M},
  date = {2012},
  eprint = {arXiv:1211.6031},
  eprintclass = {math-ph},
  doi = {10.48550/ARXIV.1211.6031}
}

@article{BBC+13,
  author = {Bostan, A and Boukraa, S and Christol, G and Hassani, S and Maillard, J-M},
  date = {2013-05-10},
  journaltitle = {Journal of Physics A: Mathematical and Theoretical},
  shortjournal = {J. Phys. A: Math. Theor.},
  volume = {46},
  number = {18},
  pages = {185202},
  doi = {10.1088/1751-8113/46/18/185202},
  title = {Ising {$n$}-Fold Integrals as Diagonals of Rational Functions and Integrality of Series Expansions}
}

@article{BBMW15,
  title = {Diagonals of Rational Functions and Selected Differential {{Galois}} Groups},
  author = {Bostan, A and Boukraa, S and Maillard, J-M and Weil, J-A},
  date = {2015-12-18},
  journaltitle = {Journal of Physics A: Mathematical and Theoretical},
  shortjournal = {J. Phys. A: Math. Theor.},
  volume = {48},
  number = {50},
  pages = {504001},
  doi = {10.1088/1751-8113/48/50/504001}
}

@article{BH89,
  shorttitle = {Monodromy for the Hypergeometric Function n {{F}} n ?},
  author = {Beukers, F. and Heckman, G.},
  date = {1989-06},
  journaltitle = {Inventiones Mathematicae},
  shortjournal = {Invent Math},
  volume = {95},
  number = {2},
  pages = {325--354},
  doi = {10.1007/BF01393900},
  title = {Monodromy for the Hypergeometric Function {${}_nF_{n-1}$}}
}

@book{But10,
  title = {Max-Linear {{Systems}}: {{Theory}} and {{Algorithms}}},
  shorttitle = {Max-Linear {{Systems}}},
  author = {Butkovič, Peter},
  date = {2010},
  series = {Springer {{Monographs}} in {{Mathematics}}},
  publisher = {Springer London},
  location = {London},
  doi = {10.1007/978-1-84996-299-5}
}

@article{BY22,
  title = {On a Class of Hypergeometric Diagonals},
  author = {Bostan, Alin and Yurkevich, Sergey},
  date = {2022},
  journaltitle = {Proceedings of the American Mathematical Society},
  shortjournal = {Proc. Amer. Math. Soc.},
  volume = {150},
  number = {3},
  pages = {1071--1087},
  doi = {10.1090/proc/15693}
}

@online{CFV25,
  title = {Galois Groups of Reductions modulo p of {{D-finite}} Series},
  author = {Caruso, Xavier and Fürnsinn, Florian and Vargas-Montoya, Daniel},
  date = {2025-05-06},
  eprint = {arXiv:2504.09429},
  eprintclass = {math},
  doi = {10.48550/arXiv.2504.09429}
}

@article{Chr86,
  title = {Fonctions hypergéométriques bornées},
  author = {Christol, Gilles},
  date = {1986/1987},
  journaltitle = {Groupe de travail d’analyse ultramétrique},
  series = {exp. no 8},
  volume = {14},
  pages = {1--16},
  url = {http://www.numdam.org/item/GAU_1986-1987__14__A4_0.pdf}
}

@article{Chr86a,
  title = {Fonctions et éléments algébriques},
  author = {Christol, Gilles},
  date = {1986-11-01},
  journaltitle = {Pacific Journal of Mathematics},
  shortjournal = {Pacific J. Math.},
  volume = {125},
  number = {1},
  pages = {1--37},
  doi = {10.2140/pjm.1986.125.1}
}

@incollection{Chr90,
  title = {Globally Bounded Solutions of Differential Equations},
  booktitle = {Analytic {{Number Theory}}},
  author = {Christol, Gilles},
  editor = {Nagasaka, Kenji and Fouvry, Etienne},
  date = {1990},
  volume = {1434},
  pages = {45--64},
  publisher = {Springer Berlin Heidelberg},
  location = {Berlin, Heidelberg},
  doi = {10.1007/BFb0097124}
}

@article{Err13,
  title = {Zahlentheoretische Lösung einer functionentheoretischen Frage},
  author = {Errera, Alfred},
  date = {1913-12},
  journaltitle = {Rendiconti del Circolo Matematico di Palermo},
  shortjournal = {Rend. Circ. Matem. Palermo},
  volume = {35},
  number = {1},
  pages = {107--144},
  doi = {10.1007/BF03015595}
}

@article{Fur67,
  title = {Algebraic Functions over Finite Fields},
  author = {Furstenberg, Harry},
  date = {1967-10},
  journaltitle = {Journal of Algebra},
  shortjournal = {Journal of Algebra},
  volume = {7},
  number = {2},
  pages = {271--277},
  doi = {10.1016/0021-8693(67)90061-0}
}

@article{FY24,
  title = {Algebraicity of Hypergeometric Functions with Arbitrary Parameters},
  author = {Fürnsinn, Florian and Yurkevich, Sergey},
  date = {2024-06-22},
  journaltitle = {Bulletin of the London Mathematical Society},
  shortjournal = {Bulletin of London Math Soc},
  pages = {blms.13103},
  doi = {10.1112/blms.13103}
}

@book{Kat90,
  title = {Exponential Sums and Differential Equations},
  author = {Katz, Nicholas M.},
  date = {1990},
  series = {Annals of Mathematics Studies},
  number = {124},
  publisher = {Princeton Univ. Press},
  location = {Princeton, N.J},
  pagetotal = {430}
}

@article{Lan04,
  title = {Eine {{Anwendung}} Des {{Eisensteinschen Satzes}} Auf Die {{Theorie}} Der {{Gaussschen Differentialgleichung}}.},
  author = {Landau, Edmund},
  date = {1904-01-01},
  journaltitle = {Journal für die reine und angewandte Mathematik (Crelles Journal)},
  volume = {1904},
  number = {127},
  pages = {92--102},
  doi = {10.1515/crll.1904.127.92}
}

@book{Lan11,
  title = {Über Einen Zahlentheoretischen {{Satz}} Und Seine {{Anwendung}} Auf Die Hypergeometrische {{Reihe}}},
  author = {Landau, Edmund},
  date = {1911},
  publisher = {Heidelberg University Library},
  doi = {10.11588/HEIDOK.00012365}
}

@article{Sch73,
  title = {Ueber diejenigen Fälle, in welchen die Gaussische hypergeometrische Reihe eine algebraische Function ihres vierten Elementes darstellt.},
  author = {Schwarz, Hermann},
  date = {1873-01-01},
  journaltitle = {Journal für die reine und angewandte Mathematik (Crelles Journal)},
  volume = {1873},
  number = {75},
  pages = {292--335},
  doi = {10.1515/crll.1873.75.292}
}

@article{Var21,
  author = {Vargas-Montoya, Daniel},
  date = {2021},
  journaltitle = {Bulletin de la Société mathématique de France},
  shortjournal = {Bul. Soc. Math. France},
  doi = {10.24033/bsmf.2834},
  title = {Algébricité modulo {$p$}, séries hypergéomé\-triques et structures de Frobenius fortes}
}

@article{Var24,
  author = {Vargas-Montoya, Daniel},
  date = {2024-06},
  journaltitle = {Journal of Number Theory},
  shortjournal = {Journal of Number Theory},
  volume = {259},
  pages = {273--321},
  doi = {10.1016/j.jnt.2024.01.004},
  title = {Algebraicity modulo {$p$} of Generalized Hypergeometric Series {${}_nF_{n-1}$}}
}

@incollection{Ked22,
 author = {Kedlaya, Kiran S.},
 title = {Frobenius structures on hypergeometric equations},
 booktitle = {Arithmetic, geometry, cryptography, and coding theory, AGC2T. 18th international conference, Centre International de Rencontres Math\'ematiques, Marseille, France, May 31 -- June 4, 2021},
 isbn = {978-1-4704-6794-4},
 pages = {133--158},
 year = {2022},
 publisher = {Providence, RI: American Mathematical Society (AMS)},
 language = {English},
 doi = {10.1090/conm/779/15673}
}

@online{github,
  author = {Caruso, Xavier and Fürnsinn, Florian},
  title = {Algebraic and modular properties of hypergeometric functions},
  date = {2025},
  url = {https://github.com/sagemath/sage/pull/41113}
}

@online{CF26-sage,
  title = {Algebraic and Arithmetic Attributes of Hypergeometric Functions in {S}age{M}ath},
  author = {Caruso, Xavier and Fürnsinn, Florian},
  date = {2026},
  eprint = {arXiv:2602.04531},
  eprintclass = {cs.SC},
  doi = {10.48550/arXiv.2602.04531}
}

@article{DRR17,
  author = {Delaygue, E. and Rivoal, T. and Roques, J.},
  date = {2017-03},
  journaltitle = {Memoirs of the American Mathematical Society},
  shortjournal = {Memoirs of the AMS},
  volume = {246},
  number = {1163},
  doi = {10.1090/memo/1163},
  title = {On {{Dwork}}’s {$p$}-Adic Formal Congruences Theorem and Hypergeometric Mirror Maps}
}

\end{document}